\documentclass[12pt,reqno]{amsart}
\usepackage{amsthm, amsmath, amssymb, amsbsy, amsfonts, latexsym, euscript,amscd}
\usepackage{calrsfs}
\usepackage{graphicx}

\def\le{\leqslant}
\def\leq{\leqslant}
\def\ge{\geqslant}
\def\geq{\geqslant}

\def\phi{\varphi}

\def\cal{\mathcal}
\def\kappa{\varkappa}

\newtheorem{theorem}{\bf \indent Theorem}[section]

\newtheorem{lemma}{\bf \indent Lemma}[section]
\newtheorem{corollary}{\bf \indent Corollary}[section]

\theoremstyle{remark}

\newtheorem{definition}{\bf \indent Definition}[section]
\newtheorem{example}{\bf \indent Example}[section]

\numberwithin{equation}{section}

\begin{document}

{\Large \bf
 CRITERIA FOR ANALYTICITY  OF\\
 
 MULTIDIMENSIONALLY\\
 
SUBORDINATE SEMIGROUPS }

\vskip7mm

{\bf A. R. Mirotin}
\vskip-2pt
{\small
 amirotin@yandex.ru}

\

\begin{flushright}
\textit{Dedicated to the memory of Yu. I. Lyubich }
\end{flushright}

\vskip10mm


Key words and phrases: strongly continuous semigroup, holomorphic semigroup, subordination, Bochner-Phillips functional calculus,  Bernstein function.

\vskip10mm

{\noindent\it
Let $\psi$ be a Bernstein function in one variable. A.~Carasso and T.~Kato obtained
necessary and sufficient conditions for $\psi$ to have a property
that $\psi(A)$ generates a quasibounded holomorphic semigroup for
every generator $A$ of a bounded $C_0$-semigroup in a Banach
space and deduced  necessary conditions as well. We generalize their results to the multidimensional case and also give  sufficient conditions for the property mentioned above.
}

{\centering
\section{Introduction}
}

The celebrated theorem of Yosida \cite{Yos60} states
that a fractional power $-(-A)^\alpha, 0<\alpha<1$ of a generator $A$ of a bounded $C_0$-semigroup on a
Banach space $X$ is a genetrator of a   holomorphic semigroup  on $X$. 
A. Carasso and T. Kato \cite[ Theorem 4]{CK} obtained a description of such  (non-negative) Bernstein functions $\psi$ in one variable
 that for every generator $A$ of a bounded $C_0$-semigroup on a
Banach space $X$ the operator $-\psi(-A)$ is a generator of a quasibounded holomorphic semigroup on $X$.  The alternative criteria   and sufficient conditions were obtained in \cite{MirSF}.
Another sufficient conditions were obtained  in  \cite{Fuj}.

 The present
paper is devoted to some generalizations and analogs of A.~ Carasso and T.~ Kato's results 
 in terms of multidimensional Bochner-Phillips calculus. \footnote{The multidimensional Bochner-Phillips calculus was developed in \cite{Mir97},  \cite{Mir98}, \cite{a&A}, \cite{MirSMZ}, \cite{IZV2015}, \cite{OaM}, \cite{OaM2}.}  These results provide a new method of constructing a large class of holomorphic 
semigroups.

\vskip7pt


{\centering
\section{Preliminaries}
}

We shall denote o by ${\cal
M}(\mathbb{R}_+^n)$ (${\mathcal M}^b(\mathbb{R}_+^n)$,  ${\mathcal M}^1(\mathbb{R}_+^n)$) the convolution algebra of all  positive
(respectively bounded positive, probability) measures on $\mathbb{R}_+^n$. Throughout, $X$ stands
for a complex Banach space,  $B(X)$ denotes the algebra of linear bounded operators on $X$, and  $I$, the identity operator on $X$. 

\begin{definition} 
 A family of operators $T(t))_{t\in \mathbb{R}_+}\subset B(X)$
 is called a  $C_0$-semigroup if

(i) $T(0) = I$ and $T(t_1 + t_2) =T(t_1 )T( t_2)$ for all $t_1, t_2 \in   \mathbb{R}_+$.

(ii) The map $t\mapsto T(t)x$ is continuous on  $ \mathbb{R}_+\forall x\in X$.

If, in addition,

(iii) 	$T(t)$	 is bounded on  $ \mathbb{R}_+$,

we call   $T(t))_{t\in \mathbb{R}_+}$ a bounded $C_0$- semigroup.
\end{definition}

In the following we denote by  $C_0(X)$ the set of bounded  strongly continuous semigroups in a complex Banach space $X$. Further,  $T_{1},\dots , T_{n}$ will denote pairwise commuting semigroups from $C_0(X)$
  with generators $A_1, \dots ,A_n$ respectively satisfying the condition $\|T_{j}(t)\|_{B(X)}\leq
M_j\quad(t\in   \Bbb{R}_+), M_j={\rm const}$.  We put also $M:=\prod_{j=1}^nM_j$.  We denote the domain of $A_j$  by $D(A_j)$ and set $A = (A_1, \dots ,A_n)$. 
By the commutation of operators $A_1, \dots, A_n$
we mean the commutation of the corresponding semigroups. By ${\rm Gen}(X)$ we denote the set of all
generators of uniformly bounded $C_0$-semigroups on $X$ and by ${\rm Gen}(X)^n$, the set of all $n$-tuples  $(A_1, \dots ,A_n)$ where $A_j\in {\rm Gen}(X)$ commute. 
An $B(X)$-valued function $T(u)=T_A(u) := T_{1}(u_1) \dots T_{n}(u_n)$   $(u\in
\Bbb{R}_+^n)$ is a bounded $n$-parameter $C_0$ semigroup with $\|T(u)\|_{B(X)}\le M$  ($u\in \Bbb{R}_+^n$).
The linear manifold $D(A):=\cap_{j=1}^nD(A_j)$ is dense in $X$ \cite[ Sec. 10.10]{HiF}.

We denote by  $\mathcal{L}\nu$  the $n$-dimensional Laplace transform
\begin{equation}\label{laplace1}
\mathcal{L}\nu(z)=\int\limits_{\mathbb{R}_+^n} e^{-z\cdot u}d\nu(u) 
\end{equation}
of a measure $\nu\in {\mathcal M}^b(\mathbb{R}_+^n)$. In this case we put also for   bounded  $n$-parameter $C_0$-semigroup $T=T_A$ and a function $g(s)=\mathcal{L}\nu(-s)$
\begin{equation}\label{laplace1}
g(A):=\langle \nu,T_A\rangle:=\int\limits_{\mathbb{R}_+^n} T_A(u)d\nu(u) 
\end{equation}
(the  Bochner  integral).
Then  $\|g(A)\|_{B(X)}\le M\|\nu\|_{ {\mathcal M}^b}$ if  $\|T_A(u)\|\le M$.

\begin{definition}
 \cite{Boch} We say that a  function $\psi\in C^\infty((-\infty;0)^n)$  {\it is a nonpositive Bernstein function of $n$ variables} 
and write  $\psi\in {\cal T}_n$  if it is nonpositive and all of its first partial derivatives are absolutely
monotone (a function in $C^\infty((-\infty;0)^n)$ is said to be absolutely monotone if it is nonnegative
together with its partial derivatives of all orders).
\end{definition}

Obviously,  $\psi\in{\cal T}_n$ if and only if $-\psi(-s)$ is a nonnegative Bernstein function of $n$ variables on $(0,\infty)^n$, and ${\cal T}_n$  is a cone under the pointwise addition of functions and multiplication by scalars. Moreover ${\cal T}_n$  is an operada  \cite[p. 321]{a&A}.
As is known \cite{Boch} (see also \cite{Mir99}, \cite{mz}), each function  $\psi\in {\cal T}_n$  admits an integral representation of the form (here
 the dot  denotes inner product in  $\mathbb{R}^n$ and the expression $s\to
-0$ means that $s_1\to -0, \ldots, s_n\to -0$)
\begin{equation}\label{psi}
\psi(s)=c_0+c_1\cdot s+\int\limits_{\Bbb{R}_+^n\setminus \{0\}}
(e^{s\cdot v}-1)d\mu (v)\quad   (s\in(-\infty;0)^n),  
\end{equation}
where $c_0=\psi(-0):=\lim\limits_{s\to -0}\psi(s)$, $c_1=(c_1^j)_{j=1}^n\in \Bbb{R}_+^n$, $c_1^j=\lim\limits_{s_j\to -\infty}\psi(s)/s_j$, and $\mu$ is a positive measure on
$\Bbb{R}_+^n\setminus \{0\}$;  $\mu$
are determined by $\psi$.

A lot of  examples of (positive) Bernstein function of one variable one can found in \cite{SSV}
(see also \cite{MirSF}, \cite{mz}). Now we shall  describe a family of functions from  ${\cal T}_2$
(for the multidimensional generalization see \cite[Lemma 1]{OaM2}).

\begin{example}\cite[Theorem 1.11]{a&A}
Let a function $\psi_1\in {\mathcal T}_1$ have the
integral representation
$$
\psi_1(s)=c_0+\int\limits_0^\infty (e^{sv}-1)d\mu_1(v),
$$
where $\omega:=\psi_1'(0-)\ne\infty$. Then the function
$$
\psi(s_1,s_2):=\frac{\psi_1(s_1)-\psi_1(s_2)}{s_1-s_2}-\omega
$$
  (under the value of the function $\psi$ for $s_1=s_2$ we, as
usually  understand its limit at $s_2\to s_1$) belongs to ${\mathcal T}_2$ and has the integral representation
$$
\psi(s_1,s_2)=\int\limits_{\mathbb{R}_+^2\setminus \{0\}}
(e^{s_1u_1+s_2u_2}-1) d\mu(u_1,u_2),
$$
\noindent where $d\mu(u_1,u_2)$ is the image of the measure $d\mu_1(v)dw$ under the map
$$
u_1=\frac{v+w}{2},\quad u_2=\frac{v-w}{2}
$$
(in particular if $d\mu_1(v)=p(v)dv$ then $d\mu(u_1,u_2)=p(u_1+u_2)du_1du_2$).
\end{example}

Each function $\psi\in{\mathcal T}_n$ according to the formula
$$
\psi(z)=c_0+c_1\cdot z+\int\limits_{\mathbb{R}_+^n\setminus \{0\}}
(e^{z\cdot u}-1)d\mu (u) 
$$
extends uniquely to a function that is holomorphic in the domain
$$\{{\rm Re}z<0\}:=\{{\rm Re}z_j<0,j=1,
\ldots,n\}\subset\mathbb{C}^n
$$
  and continuous on its closure \cite{a&A}.

\begin{definition}
   \cite{Mir99} The value of a function  $\psi\in {\cal T}_n$ of the form \eqref{psi} at $A=(A_1,\ldots ,A_n)$ applied
to $x\in D(A)$ is defined by
\begin{equation*}
\psi(A)x=c_0x+c_1\cdot Ax+\int\limits_{\Bbb{R}_+^n\setminus \{0\}}
(T_A(u)-I)xd\mu(u),     
\end{equation*}
 where the integral is taken in the sense of Bochner,  and $c_1\cdot Ax:=\sum_{j=1}^nc_1^jA_jx$.
\end{definition}

Given $\psi\in{\cal T}_n$ and $t\geq 0$, the function $g_t(z):=e^{t\psi(z)}$ is absolutely monotone on $(-\infty;0)^n$. It
is also obvious that $g_t(z)\leq 1$   on $(-\infty;0)^n$. By virtue of the multidimensional version of the Bernstein-Widder
Theorem (see, e.g.,  \cite{Boch}, \cite{BCR}), there exists a unique sub-probability measure $\nu_t$ on $\Bbb{R}_+^n$, such that, for $z\in (-\infty;0)^n$, we have
\begin{equation}\label{laplace}
g_t(z)=\int\limits_{\Bbb{R}_+^n} e^{z\cdot u}d\nu_t(u) =\mathcal{L}(\nu_t)(-z). 
\end{equation}
If $\psi(0)=0$, the family $(\nu_t)_{t\ge 0}$, where $\nu_0:=\delta_0$  (the Dirac measure) is convolution semigroup of probability measures. 

The following theorem describes the family of measures in \eqref{laplace}.

Recall that a net  $(\mu_n)\subset {\mathcal M}^b(\mathbb{R}_+^n)$ converges weakly to  $\mu\in {\mathcal M}^b(\mathbb{R}_+^n)$ (we write $\mu=w\lim_n\mu_n$)
if for each $f\in C^b(\mathbb{R}_+^n)$ 
$$
\lim_n\int_{\mathbb{R}_+^n}fd\mu_n=\int_{\mathbb{R}_+^n}fd\mu.
$$

\begin{theorem}\label{weakcont}
For  the family $(\nu_t)_{t\ge 0}\subset  \mathcal{ M}^1(\mathbb{R}_+^n)$ the following statements are equivalent:

(i)  the mapping
$t\mapsto \nu_t$, $\mathbb{R}_+\to \mathcal{ M}^b(\mathbb{R}_+^n)$ is weakly  continuous;

(ii) $\mathcal{L}\nu_t(-z)=e^{t\psi(z)}$  for some function  $\psi\in \mathcal{T}_n$, $\psi(0)=0$ ($t\in \mathbb{R}_+$, ${\rm Re}z\le 0$).
\end{theorem}

\begin{proof}
(i)$\Rightarrow$ (ii). Let  the mapping
$t\mapsto \nu_t$ be weakly continuous.
Let $g_t(s):=(\mathcal{L}\nu_t)(-s)$. Then for any $a>0$
the equality $g_{at}(s)=(g_t(s))^a$ is true. In fact, for every
$m\in\mathbb{N}$
$$
g_{mt}(s)=\int\limits_{\mathbb{R}^n_+} e^{s\cdot u}d\nu_{mt}(u)=
\int\limits_{\mathbb{R}^n_+} e^{s\cdot u}d\nu_t^{\ast m}(u)=(g_t(s))^m,
$$
\noindent
where $\nu_t^{\ast m}$ denotes the $m$th convolutional degree.
 Replacing $t$ in the last equality with $(k/m)t$, we get,
that for any positive rational $r=k/m, k,m\in \mathbb{N}$, we have $g_{rt}(s)=(g_t(s))^r$. If now the sequence
of positive rational numbers $r_k$ converges to $a$, then due to
of our continuity condition $g_{at}(s)=\lim_{k\to\infty}
g_{r_kt}(s)=(g_t(s))^a$ (see, for example, \cite[Chapter IX, \S 5, Proposition 14]{Burb}). In particular, $g_t(s)=(g_1(s))^t$, and therefore
the function $(1/t){\rm log}g_t(s)={\rm log}g_1(s)$ does not depend on $t$.
Further, a completely monotonic function $g_1(-r)= ({\mathcal L}\nu_1)(r)$, as
was shown above, satisfies the equality
$g_1(-r)^{1/m}=g_{1/m}(-r), m\in \mathbf{N}, r\in (0;\infty)^n$, and
therefore it is infinitely divisible. Consequently (see, e.~g., \cite{umzh})
 $g_1(s)=e^{\psi(s)}$, where $\psi\in \mathcal{ T}_n$, $\psi(0)=0$. 
 Thus,  $g_t(s)=e^{t\psi(s)}$.

(ii)$\Rightarrow$ (i). If  $\mathcal{L}\nu_t(-z)=e^{t\psi(z)}$, and $t_0\ge 0$, then   $\lim_{t\to  t_0}\mathcal{L}\nu_t(-z)=\mathcal{L}\nu_{t_0}(-z)$ for all $z$ with ${\rm Re}z\le 0$.
Therefore by the continuity theorem for the Laplace transform we have  $w\lim_{t\to t_0}\nu_t=\nu_{t_0}$. This completes the proof.
\end{proof}

\begin{definition} In the notation introduced above, we set
\begin{equation}\label{gA}
g_t(A)x:=\langle \nu_t,T_A x \rangle=\int\limits_{\Bbb{R}_+^n} T_A(u)xd\nu_t(u)\ (x\in X). 
\end{equation}
\end{definition}

Obviously, $\|g_t(A)\|_{B(X)}\leq M$.  The map $g(A):t\mapsto g_t(A)$ is a $C_0$-semigroup. In the one-dimensional case, it is called the
\textit{semigroup subordinate to} $T_A$ (in a sense of Bochner). In the general case, we call this semigroup \textit{multidimansionally subordinate to} $T_A$.

In \cite{Mir09} it was proved that the closure of the operator $\psi(A)$ exists and is the generator of the $C_0$-semigroup
$g(A)$ (cf. \cite{Mir99}.) Since it is no easy to find this statement in the literature, we sketch its proof
below.

\begin{theorem} \cite{Mir09} The closure of the operator  $\psi(A)$ exists and is the generator of the semigroup $g(A)$ of class $C_0$ defined by formula \eqref{gA}.
\end{theorem}

\begin{proof} It was proved in \cite{a&A} that the operator $\psi(A)$ is closable, and
its extension is the generator $G$  of the $C_0$-semigroup
$g(A)$ . Therefore, the operator $G$ is an extension of the closure  $\overline{\psi(A)}$ of $\psi(A)$. We show that an equality holds here. Since the operators $g_t(A)$ commute with  $T_j(s)$ for all $j$, it is easy to check that $g_t(A):D(A_j)\to D(A_j)$ ) for all $j$, and therefore $g_t(A):D(A)\to D(A)$.
It follows that $D(A)$ is an operator core for  the generator $G$
(see, e.~g.,  \cite[Corollary 3.1.7]{BrRob}). On the other hand, $D(A)$ is an operator core  for the operator $\psi(A)$, and the restriction of the operators $G$ to $D(A)$
coincide with  $\psi(A)$. Therefore $\overline{\psi(A)}=G$, which completes the proof.
\end{proof}

The previous result suggests the following final version of the definition of the operator $\psi(A)$.

\begin{definition}
\cite{Mir99} By the value of a function $\psi\in {\cal T}_n$ at an $n$-tuple $A = (A_1,\dots ,A_n)$ of
commuting operators in ${\rm Gen}(X)$ we understand the generator of the semigroup $g(A)$, i.e.,  the closure of the operator defined in the Definition \ref{psi}. This value is
denoted by $\psi(A)$.
\end{definition}

The functional calculus thus arising is called multidimensional Bochner-Phillips
calculus, or $\mathcal{T}_n$-calculus.

Without loss of generality we shall assume that
$c_0=0$ in \eqref{psi}. We assume also that $c_1=0$.The corresponding subclass of  ${\cal T}_n$ will be denoted
by ${\cal T}^{0}_n$. 

The notation and constraints introduced above are used in what follows without additional
explanations.

\begin{definition} Let
\begin{equation*}
\Sigma(\theta):=\{z\in \mathbb{ C}:\mathrm{Re}(z)>0, |\arg(z)|< \theta<\pi/2\}.
\end{equation*}

 A family of operators $(V(z))_{z\in \Sigma(\theta)\cup\{0\}}\subset B(X)$
 is called a  holomorphic
(analytic) semigroup (of angle $\theta\in  (0, \pi/2]$) if

(i) $V(0) = I$ and $V(z_1 + z_2) = V(z_1)V(z_2)$ for all $z_1, z_2 \in  \Sigma(\theta)$.

(ii) The map $z\mapsto V(z)$ is analytic in $\Sigma(\theta)$.

(iii)
$\underset{z\in \Sigma(\theta')\ni z\to 0}{\lim}V(z)x=x$
  $\forall x\in X$ and $0<\theta'<\theta$.

If, in addition,

(iv) 	$V(z)$	 is bounded in $\Sigma(\theta')$ for every $0<\theta'<\theta$,

we call   $(V(z))_{z\in \Sigma(\theta)\cup\{0\}}$ a bounded holomorphic semigroup.
\end{definition}

A result by Yosida \cite{Yos59} asserts that if the bounded
$C_0$-semigroup $U$ with generator $G$ on  $X$ satisfies
\begin{equation*}
U(t)X\subset D(G), t>0,\quad {\rm and} \quad \lim\sup\limits_{t\to 0+} (t\|GU(t)\|)<
\infty, \eqno(Y)
\end{equation*}
then for any $\beta>0$, $e^{-\beta t}U(t)$ can be extended to a bounded  holomorphic
semigroup  on $X$. In this case, the semigroup $U$ will be called {\it a quasibounded holomorphic semigroup}. 

We shall denote by  ${\cal T}_n^Y$   the set of all $\psi\in{\cal T}_n^0$  such that
$\psi(G)$ generates  a bounded  $C_0$-semigroup   with property (Y) for every generator
$G$ of     a bounded  $C_0$-semigroup in a Banach space. 

{\centering
\section{Criteria for analyticity }
}

In the sequel, for any measure $\mu\in {\mathcal M}^b(\mathbb{R}_+^n)$  we put
\begin{equation*}
(Z_\mu f)(x):=(\mu\ast f)(x)= \int_{\mathbb{R}_+^n}f(x-u)d\mu(u) ,\quad f\in L^1(\mathbb{R}^n_+).
\end{equation*}
Note, that the semigroup $Z_{\nu_t}$ is of the form $\langle \nu_t,T\rangle$ where $T(u)f(x)=f(x-u)$ ($u\in \mathbb{R}^n_+$) is the left translation  $n$-parameter semigroup in  $L^1(\mathbb{R}^n_+)$.

For the convolution semigroup $(\nu_t)_{t\ge 0}\subset {\mathcal M}^1(\mathbb{R}_+^n)$ we define the weak derivative as 
$$
w\nu'_t:=w\lim_{\Delta t\to 0}\frac{1}{\Delta t}(\nu_{t+\Delta t}-\nu_t).
$$

\begin{theorem}\label{criteria} (cf. \cite{CK}). Let  $\psi\in  {\cal T}_n^0$. 
The following conditions are equivalent:

(a) for all $T_j\in C_0(X)$ ($j=1,\dots,n$) the semigroup $g_t(A)=\langle \nu_t,T\rangle$ satisfies the Yosida condition (Y) for  every  Banach space $X$ (i.e. $\psi(A)$ is a generator of   a quasibounded holomorphic semigroup $g(A)$ in $X$);

(b)  the semigroup $Z_{\nu_t}$ satisfies the Yosida condition (Y) in the  space $L^1(\mathbb{R}^n_+)$;

(c) the map $t\mapsto \nu_t$, $\mathbb{R}_+\to {\mathcal M}^b(\mathbb{R}_+^n)$ is continuously differentiable for $t>0$, with $\|\nu_t'\|_{{\mathcal M}^b}=O(t^{-1})$ as $t\to 0+$.

(d) the map $t\mapsto \nu_t$, $\mathbb{R}_+\to {\mathcal M}^b(\mathbb{R}_+^n)$ is weakly differentiable for $t>0$, with $\|w\nu'_t\|_{{\mathcal M}^b}=O(t^{-1})$ as $t\to 0+$.

\end{theorem}

\begin{proof}
$(a) \Rightarrow (b)$ This is obvious, since $Z_{\nu_t}=\langle \nu_t,T\rangle$ where $T$ stands for the left translation  semigroup in  $L^1(\mathbb{R}^n_+)$.

$(b) \Rightarrow (c)$  Since  $Z_{\nu_t}$ satisfies the  condition (Y) in  $L^1(\mathbb{R}^n_+)$, the semigroup $U(t)=e^{-\beta t}Z_{\nu_t}$ is holomorphic n  $L^1(\mathbb{R}^n_+)$ by the Yosida theorem. Then the semigroup $Z_{\nu_t}$ is norm differentiable in  $L^1(\mathbb{R}^n_+)$ for $t>0$, as well.
Due to the isometric isomorphism $\nu\mapsto Z_\nu$, ${\mathcal M}^b(\mathbb{R}_+^n)\to B(L^1(\mathbb{R}^n_+))$ we conclude that 
\begin{equation}\label{deriv}
\nu'_t=\lim\limits_{\Delta t\to 0+}\frac{\nu_{t+\Delta t}-\nu_t}{\Delta t}
\end{equation}
in the  ${\mathcal M}^b(\mathbb{R}_+^n)$ norm and 
\begin{equation*}
\|\nu'_t\|_{{\mathcal M}^b}=\|Z_{\nu'_t}\|_{B(L^1)}.
\end{equation*}
Moreover, since the semigroup $Z_{\nu_t}$ satisfies the Yosida condition (Y), we have
$\limsup\limits_{t\to 0+}(t\|\nu'_t\|_{{\mathcal M}^b})<\infty$,
and thus $\|\nu_t'\|_{{\mathcal M}^b}=O(t^{-1})$ as $t\to 0+$.

Now since $\nu_t'\in {\mathcal M}^b(\mathbb{R}_+^n)$ for all $t>0$, we have  $\nu_{t/2}'\ast \nu_{t/2}'\in {\mathcal M}^b(\mathbb{R}_+^n)$, too. 
Taking into account that $({\mathcal L}\nu_t)(-z)=e^{t\psi(z)}$ we deduce that  $\nu_{t/2}'\ast \nu_{t/2}'=\nu_t^{\prime\prime}$ (the Laplace transform of the both sides of the last equality is equal to $\psi(z)^2e^{t\psi(z)}$). So $\|\nu_t^{\prime\prime}\|_{{\mathcal M}^b}\le \|\nu_{t/2}^{\prime}\|_{{\mathcal M}^b}^2$. In particular, the function $t\mapsto \|\nu_{t}^{\prime}\|_{{\mathcal M}^b}$ is bounded and continuous on  each interval $(t_1,t_2)\subset \mathbb{R}_+$.

$(c) \Rightarrow (d)$ This is obvious.

$(d) \Rightarrow (a)$ Note that for every $x\in X$ and every continuous linear functional $l\in X^*$ the function $u\mapsto l(T(u)x)$ belongs to $C^b(\mathbb{R}_+^n)$. Then 
\begin{align*}
l\left(\frac{d}{dt}g_t(A)x\right)&=\frac{d}{dt}l\left(g_t(A)x\right)=\lim_{\Delta t\to 0}\frac{1}{\Delta t}(l\left(g_{t+\Delta t}(A)x\right)-l\left(g_t(A)x\right))\\[5pt]
&=w\lim_{\Delta t\to 0}\frac{1}{\Delta t}\int\limits_{\mathbb{R}_+^n}l(T(u)x)d(\nu_{t+\Delta t}-\nu_t)(u)\\[5pt]
&=l\left(w\lim_{\Delta t\to 0}\frac{1}{\Delta t}\int\limits_{\mathbb{R}_+^n}T(u)xd(\nu_{t+\Delta t}-\nu_t)(u)\right)\\[5pt]
&=l\left(\int\limits_{\mathbb{R}_+^n}T(u)xd\left(w\nu'_t\right)(u)\right).
\end{align*}
Therefore
\begin{equation*}
\left\|\frac{d}{dt}g_t(A)x\right\|_{B(X)}=\left\|\int\limits_{\mathbb{R}_+^n}T(u)xd\left(w\nu'_t\right)(u)\right\|\le M\left\|w\nu'_t\right\|_{{\mathcal M}^b}\|x\|,\quad t>0.
\end{equation*}
Since $\left\|w\nu'_t\right\|_{{\mathcal M}^b}=O(t^{-1})$ as $t\to 0+, (a)$ follows.
\end{proof}

To compute the  weak derivative $w\nu'_t$ we introduce the following notion.
Consider the vector space of exponential polynomials on $\mathbb{R}^n_+$
$$
E(\mathbb{R}^n_+):=\left\{p(r)=\sum_jc_je^{s_j\cdot r}: c_j\in \mathbb{C}, s_j\in (-\infty,0)^n\right\}
$$
equipped  with the ${\rm sup}$ norm.

\begin{definition}\label{bt}
 For $p\in E(\mathbb{R}^n_+)$, $t>0$ let
 $$
b_t(p):= \int\limits_{\mathbb{R}_+^n} \int\limits_{\mathbb{R}_+^n}p(r)d_r(\nu_t(r-u)-\nu_t(r))d\mu(u).
$$
\end{definition}

Below we identify a finite measure on  $\mathbb{R}^n_+$ with the corresponding integral, i.~e. with the corresponding bounded linear functional on  $C_0(\mathbb{R}^n_+)$.

\begin{theorem}\label{d/dt=bt}
If the map $t\mapsto \nu_t$, $\mathbb{R}_+\to {\mathcal M}^b(\mathbb{R}_+^n)$ is weakly differentiable for $t>0$ then
the linear functional $b_t$ extends uniquely from 
$E(\mathbb{R}^n_+)$ to some bounded Radon measure on $\mathbb{R}_+^n$ and $w\nu'_t=b_t$.
\end{theorem}

\begin{proof} 
 For $p\in E(\mathbb{R}^n_+)$, $p(r)=\sum_jc_je^{s_j\cdot r}$, $t>0$ we have, since $\nu_t$ is concentrated on $\mathbb{R}^n_+$,
\begin{align*}
 \int\limits_{\mathbb{R}_+^n}p(r)d_r(\nu_t(r-u)-\nu_t(r))&= \int\limits_{\mathbb{R}_+^n}p(r)d_r\nu_t(r-u)- \int\limits_{\mathbb{R}_+^n}p(r)d_r\nu_t(r)\\[5pt]
&=\int\limits_{-u+\mathbb{R}_+^n}p(r)d_r\nu_t(r)- \int\limits_{\mathbb{R}_+^n}p(r)d_r\nu_t(r)\\[5pt]
&=\int\limits_{\mathbb{R}_+^n}(p(r+u)-p(r))d_r\nu_t(r)\\[5pt]
&=\int\limits_{\mathbb{R}_+^n}\sum_jc_j e^{s_j\cdot r}(e^{s_j\cdot u}-1)d_r\nu_t(r)\\[5pt]
&=\sum_j c_j\left( e^{s_j\cdot u}-1\right)e^{t\psi(s_j)}.
\end{align*}
Therefore
\begin{align*}
b_t(p)&=\int\limits_{\mathbb{R}_+^n}\left(\sum_j c_j\left( e^{s_j\cdot u}-1\right)e^{t\psi(s_j)}\right)d\mu(u)\\[5pt]
&=\sum_j c_je^{t\psi(s_j)}\int\limits_{\mathbb{R}_+^n}\left(e^{s_j\cdot u}-1\right)d\mu(u)\\[5pt]
&=\sum_j c_je^{t\psi(s_j)}\psi(s_j).
\end{align*}

On the other hand,
\begin{align*}
\nu_t(p)\equiv \int\limits_{\mathbb{R}_+^n}p(r)d\nu_t(r)=\sum_j c_je^{t\psi(s_j)}.
\end{align*}
It follows that 
\begin{align}\label{eq}
\frac{d}{dt}\nu_t(p)=\sum_j c_je^{t\psi(s_j)}\psi(s_j)=b_t(p).
\end{align}
Note that $ E(\mathbb{R}^n_+)$ is a dense sub-algebra of $C_0(\mathbb{R}^n_+)$ by the Stone-Weierstrass theorem. So, if  $w\nu'_t$ exists, then by \eqref{eq} the linear functional $b_t$ extends uniquely from 
$E(\mathbb{R}^n_+)$ and its extension equals to  $w\nu'_t$.


\end{proof}

\vskip7pt

{\centering
\section{The necessary  conditions}
}

For $z\in \mathbb{ C}^n$ we put $\|z\|:=\max\limits_{1\le j\le n}|z_j|$, $\mathbb{D}^n:=\{z:\|z\|<1\}$, $\Pi^n_{-}:=\{z\in \mathbb{C}^n: \mathrm{Re}z_j<0,  j=1,\dots, n\}$, and  $\overline{\Pi^n_{-}}$ the closure of $\Pi^n_{-}$.

\begin{theorem}\label{necessary} (cf. \cite{CK}).
If the conditions $(a)-(c)$ of Theorem \ref{criteria} hold then $\psi$ maps the region $\Pi^n_{-}$ into a truncated sector 
$$
S(\theta,\beta):= (\beta+\{|\arg(-z)|<\theta\})\cap \Pi^1_-.
$$
of opening $2\theta<\pi$ and there exist constants $K>0$ and $\gamma$, $\gamma\in (0,1)$, such that 
\begin{equation*}
|\psi(z)|\le K\|z\|^\gamma \quad \mbox{   {\rm for all}   }  z\in \overline{\Pi^n_{-}}, \|z\|\ge 1.
\end{equation*}

\end{theorem}

\begin{proof}
Step 1. This step use arguments that are  similar to the  arguments from the corresponding  part of the proof of Theorem 4 in \cite{CK}. Fix a number $\beta>0$. Due to the condition $(b)$
of Theorem \ref{criteria}, the semigroup  $e^{-\beta t}Z_{\nu_t}$ on  $L^1(\mathbb{R}^n_+)$ can be extended to a bounded  holomorphic
semigroup   in a sector $\Sigma(\theta)$ ($0<\theta<\pi/2$).
We denote this analytic continuation by $e^{-\beta z}Z_{\nu}(z)$. Thus, $e^{-\beta t}Z_{\nu}(t)=e^{-\beta t}Z_{\nu_t}$.

In \cite{CK} was noticed  that the condition (Y) together with $\|Z_{\nu_t}\|_{B(L^1)}\le M$ imply
\begin{equation*}
 C_\beta:=\sup\limits_{t>0}{t(\|(e^{-\beta t}Z_{\nu}(t))'\|_{B(L^1)}<\infty}.
\end{equation*}
Moreover, since   $U^{(n)}(t)=(U'(t/n))^n$   for every holomorphic semigroup $U$  (see, e.g., \cite[p. 98]{EngNag}), we have
\begin{eqnarray*}
\|(e^{-\beta t}Z_{\nu}(t))^{(n)}\|_{B(L^1)}&\le &\|(e^{-\beta t/n}Z_{\nu}(t/n))^{\prime}\|_{B(L^1)}^n\\
&\le&(nt^{-1}C_\beta)^n\le n!(et^{-1}C_\beta)^n,\quad t>0, n\ge 1.
\end{eqnarray*}
For $n=0$ the last estimate is true, too, because $\nu_t$ is a sub-probability measure.

Now fix $q\in (0,1)$, $\theta\in (0,\arctan(q/(eC_\beta))$ and consider the Taylor series   
\begin{eqnarray*}
e^{-\beta z}Z_{\nu}(z)=\sum_{n=0}^\infty \frac{1}{n!}\left(e^{-\beta \mathrm{Re}z}Z_{\nu}(\mathrm{Re}z)\right)^{(n)}(\imath \mathrm{Im}z)^n.
\end{eqnarray*}
This series converges uniformly in $B(L^1(\mathbb{R}^n_+))$ for $\mathrm{Re}z>0$, $|\arg(z)|<\theta$, since 
\begin{eqnarray*}
\left\|\frac{1}{n!}\left(e^{-\beta \mathrm{Re}z}Z_{\nu}(\mathrm{Re}z)\right)^{(n)}(\imath \mathrm{Im}z)^n\right\|_{B(L^1)}\le (eC_\beta)^n
\left(\frac{|\mathrm{Im}z|}{\mathrm{Re}z}\right)^n\le q^n.
\end{eqnarray*}
This estimate shows also that for $z\in \Sigma(\theta/2)$
\begin{eqnarray*}
\left\|e^{-\beta z}Z_{\nu}(z)\right\|_{B(L^1)}\le\frac{1}{1-q}.
\end{eqnarray*}

Now the isometric isomorphism $\nu\mapsto Z_\nu$, ${\mathcal M}^b(\mathbb{R}_+^n)\to B(L^1(\mathbb{R}^n_+))$ implies that the map  $t\mapsto\nu_t$, $(0,\infty)\to {\mathcal M}^b(\mathbb{R}_+^n)$ has the analytic continuation  $\zeta\mapsto\nu(\zeta)$,  $\Sigma(\theta)\to {\mathcal M}^b(\mathbb{R}_+^n)$ and
\begin{eqnarray*}
\|e^{-\beta \zeta}\nu(\zeta)\|_{{\mathcal M}^b}\le\frac{1}{1-q},\quad \zeta\in \Sigma(\theta).
\end{eqnarray*}
Note that by analyticity  \eqref{laplace} is valid for $t\in \Sigma(\theta)$. The application of the Laplace transform to the  measure $e^{-\beta t}\nu(t)$ yields 
\begin{eqnarray}\label{est}
|e^{-t(\beta -\psi(w))}|\le\|e^{-\beta t}\nu(t)\|_{{\mathcal M}^b}\le\frac{1}{1-q},\quad w\in \overline{\Pi^n_{-}}, \  t\in \Sigma(\theta).
\end{eqnarray}

 Let $\overline{\Pi^n}:=-\overline{\Pi^n_{-}}$,  $\Pi^n:=-\Pi^n_{-}$, and $\xi(z):=\beta -\psi(-z)$, where $z\in \overline{\Pi^n}$. The inequality \eqref{est} implies that $\xi$ maps 
$\Pi^n$ into $\Sigma((\pi/2-\theta)$ (the boundedness of $t\mapsto e^{-t\zeta}$ on the sector $\Sigma(\theta)$ implies $\zeta\in \Sigma((\pi/2-\theta)$). It follows that  $\psi$ maps the region $\Pi^n_{-}$ into a truncated sector $(\beta- \Sigma((\pi/2-\theta))\cap \overline{\Pi_1^-}$ of opening $<\pi$.

Step 2. Now let
\begin{eqnarray*}
f(z):=\left(\prod_{j=1}^nz_j^{\frac{2\theta}{\pi n}}\right)\xi(z)/\xi(1,\dots, 1),
\end{eqnarray*}
and
\begin{eqnarray*}
z=(z_j)_{j=1}^n:=\left(\frac{1+w_j}{1-w_j}\right)_{j=1}^n,\ h(w):=f(z),\ g(w)=\frac{h(w)-1}{h(w)+1},\quad w=(w_j)_{j=1}^n\in \mathbb{D}^n.
\end{eqnarray*}
Since $f:\Pi^n\to\Pi^1$, $f(1,\dots,	1)=1$, we get $h:\mathbb{D}^n\to \Pi^1$, $h(0)=1$, and thus $g:\mathbb{D}^n\to \mathbb{D}$, $g(0)=0$.
By the Scwatrz Lemma
\begin{eqnarray*}
|g(w)|\le\|w\|,\quad w\in \mathbb{D}^n.
\end{eqnarray*}
But we have $h(w)=(1+g(w))/(1-g(w))$, and so 
\begin{eqnarray*}
|f(z)|=|h(w)|\le \frac{1+|g(w)|}{1-|g(w)|}\le  \frac{1+\max_j|w_j|}{1-\max_j|w_j|},\quad w\in \mathbb{D}^n.
\end{eqnarray*}

In other words,
\begin{eqnarray*}
|f(z)|\le   \frac{1+\max_j\left|\frac{z_j-1}{z_j+1}\right|}{1-\max_j\left|\frac{z_j-1}{z_j+1}\right|},\quad z\in \Pi^n.
\end{eqnarray*}

If we consider $z_j=x_j\ge 1$ for  all $j=1,\dots,n$, then $f(x)>0$ and 
\begin{eqnarray*}
f(x)\le   \frac{1+\max_j\frac{x_j-1}{x_j+1}}{1-\max_j\frac{x_j-1}{x_j+1}}=\frac{1+\frac{x_k-1}{x_k+1}}{1-\frac{x_k-1}{x_k+1}}=x_k
\end{eqnarray*}
for some $k$, $1\le k\le n$. Hence, $0<f(x)\le \max_k x_k$. This inequality means that for all $\beta>0$
\begin{eqnarray*}
\prod_{j=1}^nx_j^{\frac{2\theta}{\pi n}}(\beta-\psi(-x))\le (\beta-\psi(-1,\dots,-1))\max_k x_k
\end{eqnarray*}
Putting here $\beta\to 0$, we  get
\begin{eqnarray}\label{real}
-\psi(-x)\le A(\max_k x_k)\prod_{j=1}^nx_j^{-\frac{2\theta}{\pi n}}\le A (\max_k x_k)(\max_k x_k)^{-\frac{2\theta}{\pi}}=A\|x\|^{\gamma},
\end{eqnarray}
where $A=-\psi(-1,\dots,-1)$, $\gamma=1-\frac{2\theta}{\pi}\in (0,1)$, and all $x_j$ are real, $\|x\|\ge 1$.

Step 3. To finish the proof, we shall use the representation \eqref{psi} with $c_0=c_1=0$ to obtain the similar estimate for $z\in \overline{\Pi^n_-}$. 

First note that for $a, u\in \mathbb{R}_+^n$, $\sigma=1-e^{-1}$
\begin{equation*}
1-e^{-a\cdot u}\ge \sigma,\  \mbox{ as } a\cdot u\ge 1,\quad  1-e^{-a\cdot u}\ge \sigma a\cdot u \mbox { as } a\cdot u\le 1.
\end{equation*}
Therefore,  by \eqref{est} we have for $a\in \mathbb{R}_+^n$ with $\|a\|\ge 1$ 
\begin{align*}
A\|a\|^{\gamma}&\ge -\psi(-a)=\int_{\mathbb{R}_+^n}(1-e^{-a\cdot u})d\mu(u)\\[5pt]
&=\left(\int_{a\cdot u\le 1}+\int_{a\cdot u> 1}\right)(1-e^{-a\cdot u})d\mu(u)\\[5pt]
&\ge \sigma \int_{\{a\cdot u\le 1\}}a\cdot u d\mu(u)+\sigma \int_{\{a\cdot u> 1\}} d\mu(u).
\end{align*}
In particular, if we take $\forall j a_j=c$, $c\ge 1$, then
\begin{equation*}
Ac^{\gamma}\ge \sigma c\int_{\{c\sum_ju_j\le 1\}}\sum_{j=1}^nu_j d\mu(u)+\sigma \int_{\{c\sum_ju_j> 1\}} d\mu(u).
\end{equation*}
It follows that for all $c\ge 1$
\begin{equation}\label{36}
\int_{\{c\sum_ju_j\le 1\}}\sum_{j=1}^nu_j d\mu(u)\le \frac{A}{\sigma}c^{\gamma-1},\quad \int_{\{c\sum_ju_j> 1\}} d\mu(u)\le \frac{A}{\sigma}c^{\gamma}.
\end{equation}

Next, since for all  $z\in \overline{\Pi^n_-}$, $u\in \mathbb{R}_+^n$ we have
\begin{equation*}
\left|1-e^{z\cdot u}\right|\le\min\{z\cdot u,2\}\le\min\{\|z\|\sum_{j=1}^n u_j,2\}.
\end{equation*}
This estimate and \eqref{36} imply for all $z\in \overline{\Pi^n_-}$, $c\ge 1$
\begin{align*}
|\psi(z)|&\le \int_{\mathbb{R}_+^n}|1-e^{z\cdot u}|d\mu(u)\\[5pt]
&=\left(\int_{\{c\sum_ju_j\le 1\}}+\int_{\{c\sum_ju_j> 1}\right)|1-e^{z\cdot u}|d\mu(u)\\[5pt]
&\le \|z\|\int_{\{c\sum_ju_j\le 1\}}\sum_{j=1}^nu_j d\mu(u)+2\int_{\{c\sum_ju_j> 1\}} d\mu(u)\\[5pt]
&\le\|z\|\frac{A}{\sigma}c^{\gamma-1}+2\frac{A}{\sigma}c^{\gamma}.
\end{align*}
Putting here $c=\|z\|$, we get $|\psi(z)\le 3\frac{A}{\sigma}\|z\|^\gamma$  for $\|z\|\ge 1$. This completes the proof.
\end{proof}

{\centering
\section{Sufficient conditions}
}

It is known that the condition
\begin{equation*}
\limsup\limits_ {t \to 0+} \|I-T(t) \| <2
\end{equation*}
is sufficient for the analyticity of a one parametric
$ C_0$-semigroup $T $ in a Banach space $ X$. Although in general
the converse  fails, this   condition  is necessary for
analyticity of $ T$ if $X $ is locally convex.

The next theorem from \cite{SMZ2011} generalizes this result.

\begin{theorem}
Assume that $ C_0 $-semigroups
$ T_j $ satisfy the conditions  $ \|T_j (t) \| \leq M_j \ (j = 1, \dots,
n) $ and
\begin{equation*}
\sum\limits_ {j = 1}^n C_{j} \limsup\limits_ {t \to
0+} \| I-T_j (t) \| <2,
\end{equation*}
where $ C_j = \prod_{k = 1}^{j-1} M_k $ as $ j> 1$, $C_1 =  1$. Then for each
function $ \psi \in \mathcal{T}_n $ the operator $ \psi(A) $ is a
generator of a holomorphic semigroup.
\end{theorem}

We are going to deduce several sufficient conditions for $\psi\in \mathcal{T}^Y_n$ from Theorem  \ref{criteria}.

In the following we shall denote  by $\mathcal{F}$ the Fourier transform on
$\mathbb{R}^n$ in a sense of distributions,
 and by  ${\cal F}^{-1}$ the inverse of ${\cal F}$. Let
$$
F_t(\lambda)=e^{t\psi(\imath \lambda)}\psi(\imath \lambda)\quad ({\rm Im}\lambda_j\geq 0, t>0).
$$
The restriction $F_t|\mathbb{R}^n$  will be so denoted by $F_t$, too. We put also
$$
J_t(\psi)=\int\limits_{\mathbb{R}^n_+}\left|(\mathcal{F}F_t)(u)\right|du
$$
(if this expression makes sense).

\begin{theorem}\label{sufficient1} Let  the Fourier transform 
$p_t=\mathcal{F}(F_t)$) belongs to $L^1(\mathbb{R}^n)$ and is concentrated on $\mathbb{R}^n_+$. 
If
\begin{align*}
\limsup\limits_{t\to 0+}tJ_t(\psi)<\infty,
\end{align*}
then $\psi\in \mathcal{T}^Y_n$ (i.e. $\psi(A)$ is a generator of a quasibounded holomorphic semigroup $g(A)$ in $X$).
\end{theorem}

\begin{proof} Since  $p_t\in L^1(\mathbb{R}^n_+)$ we have by the inverse formula for  the Fourier transform 
\begin{equation*}
F_t(y)=\frac{1}{(2\pi)^{n/2}}\int\limits_{\mathbb{R}^n_+}p_t(u)e^{\imath y\cdot u}du
\end{equation*}
for a.~e. $y\in \mathbb{R}^n$.  By the continuity the
last equality holds for all $y\in \mathbb{R}^n$. Therefore we have
for the Laplace transform
\begin{equation}\label{v}
\mathcal{L}p_t(-z)=\int\limits_{\mathbb{R}^n_+}e^{z\cdot r}p_t(r)dr=(2\pi)^{n/2}e^{t\psi(z)}\psi(z),\quad
{\rm Re}z_j\leq 0,
\end{equation}
by the Pinchuk boundary uniqueness theorem  \cite{Pinchuk} because  both sides here are analytic on the domain $\Pi^n_-$, continuous on its closure, and coincide on $\imath\mathbb{R}^n$.

On the other hand, 
$$
\mathcal{L}\nu_t(-z)=\int\limits_{\mathbb{R}^n_+}e^{z\cdot r}d\nu_t(r)=e^{t\psi(z)},\quad
{\rm Re}z_j\leq 0.
$$
It follows that
\begin{equation}\label{vv}
e^{t\psi(z)}\psi(z)=\frac{d}{dt}\int\limits_{\mathbb{R}^n_+}e^{z\cdot r}d\nu_t(r)=\int\limits_{\mathbb{R}^n_+}e^{z\cdot r}d(w\nu'_t)(r).
\end{equation}
Comparing \eqref{v} and \eqref{vv} we conclude that
$$
d(w\nu'_t)(r)=(2\pi)^{n/2}p_t(r)dr.
$$
Consequently,
$$
\|w\nu'_t\|_{\mathcal{M}^b}=(2\pi)^{n/2}\|p_t\|_{L^1}=J_t(\psi),
$$
and the result follows from Theorem \ref{criteria}.
\end{proof}

Let
 $$
K(\nu_t,\mu):=\sup\limits_{\|p\|_{E(\mathbb{R}_+^n)}=1} \left|\int\limits_{\mathbb{R}_+^n} \int\limits_{\mathbb{R}_+^n}p(r)d_r(\nu_t(r-u)-\nu_t(r))d\mu(u)\right|.
$$

\begin{theorem}\label{Knutmu}
Let
\begin{align*}
\limsup\limits_{t\to 0+}tK(\nu_t,\mu)<\infty.
\end{align*}
If the map $t\mapsto \nu_t$, $\mathbb{R}_+\to {\mathcal M}^b(\mathbb{R}_+^n)$ is weakly differentiable for $t>0$, then $\psi\in \mathcal{T}^Y_n$ (i.e. $\psi(A)$ is a generator of a quasibounded holomorphic semigroup $g(A)$ in $X$).
\end{theorem}

\begin{proof} Since $E(\mathbb{R}_+^n)$ is dense in $C_0(\mathbb{R}_+^n)$ by the Stone-Weierstrass theorem, we have $\|b_t\|=K(\nu_t,\mu)$ (see Definition \ref{bt}). On the other hand, $w\nu'_t=b_t$ by Theorem \ref{d/dt=bt}, and the result follows from Theorem  \ref{criteria}.

\end{proof}

We shall deduce one more sufficient condition for $\psi\in \mathcal{T}^Y_n$ from Theorem  \ref{criteria}
 using the following lemma.

\begin{lemma}\label{lemma1}
Let $f\in L^1(\mathbb{R}_+^n)$, $g(s)=\mathcal{L}f(-s)$, and $\mu\in {\cal
M}(\mathbb{R}_+^n)$.  Assume that
\begin{equation*}
k(f,\mu):=\int_{\mathbb{R}_+^n}\int_{\mathbb{R}_+^n}|f(r-u)-f(r)|drd\mu(u)<\infty.
\end{equation*}
Then

1) for every $\psi\in \mathcal{T}^0_n$ there is such $b\in L^1(\mathbb{R}_+^n)$ that
\begin{equation*}
h(s):=\psi(s) g(s)=\mathcal{L}b(-s), \mbox{  and  } \|b\|_{L^1}\le k(f,\mu);
\end{equation*}

2) for every $A\in \mathrm{Gen}(X)^n$ with $\|T_A\|_{B(X)}\le M  \mbox{  we have  }$ 
\begin{equation*}
h(A)=\psi(A) g(A) \mbox{  and  } \|h(A)\|_{B(X)}\le Mk(f,\mu).
\end{equation*}

\end{lemma}

\begin{proof}
1) Since   $f=0$  on $\mathbb{R}^n\setminus \mathbb{R}_+^n$,  we have for $u\in \mathbb{R}_+^n$, $s\in(-\infty,0)^n$
\begin{align*}
(e^{s\cdot u}-1)g(s)&=\int_{\mathbb{R}_+^n}e^{s\cdot (u+r)}f(r)dr-\int_{\mathbb{R}_+^n}e^{s\cdot r}f(r)dr\\[5pt] 
&=\int_{\mathbb{R}_+^n}e^{s\cdot r}(f(r-u)-f(r))dr.
\end{align*}
Then 
\begin{align*}
h(s)&=\psi(s) g(s)=\int_{\mathbb{R}_+^n}(e^{s\cdot u}-1)g(s)d\mu(u)\\[5pt]  
&=\int_{\mathbb{R}_+^n}\int_{\mathbb{R}_+^n}e^{s\cdot r}(f(r-u)-f(r))dr\mu(u)\\[5pt]  
&=\int_{\mathbb{R}_+^n}e^{s\cdot r}\int_{\mathbb{R}_+^n}e^{s\cdot r}(f(r-u)-f(r))\mu(u)dr\\[5pt] 
&=\mathcal{L}b(-s),
\end{align*}
where $b(r):=\int_{\mathbb{R}_+^n}(f(r-u)-f(r))\mu(u)$, $\|b\|_{L^1}\le k(f,\mu)$.
Above (and below)  the  application of Fubini theorem is justified, since  $k(f,\mu)<\infty$.

2)  For  $A\in \mathrm{Gen}(X)^n$ consider the Bochner integral
\begin{align*}
h(A)&=\int_{\mathbb{R}_+^n}T_A(r)b(r)dr\\[5pt] 
&=\int_{\mathbb{R}_+^n}T_A(r)\int_{\mathbb{R}_+^n}(f(r-u)-f(r))d\mu(u)dr\\[5pt] 
&=\int_{\mathbb{R}_+^n}\left(\int_{\mathbb{R}_+^n}(f(r-u)-f(r))T_A(r)dr\right)d\mu(u)\\[5pt] 
&=\int_{\mathbb{R}_+^n}\int_{\mathbb{R}_+^n}f(r-u)T_A(r)drd\mu(u)-\int_{\mathbb{R}_+^n}\int_{\mathbb{R}_+^n}f(r)T_A(r)drd\mu(u).
\end{align*}

Furthermore, since $f$ is concentrated on $\mathbb{R}_+^n$, 
\begin{align*}
\int_{\mathbb{R}_+^n}f(r-u)T_A(r)dr&=\int_{\mathbb{R}_+^n}f(r)T_A(r+u)dr\\[5pt]
&=T_A(u)\int_{\mathbb{R}_+^n}f(r)T_A(r)dr.
\end{align*}
Then
\begin{align*}
h(A)=\int_{\mathbb{R}_+^n}(T_A(u)-I)\int_{\mathbb{R}_+^n}f(r)T_A(r)drd\mu(u)=\psi(A)g(A).
\end{align*}

Finally, 
\begin{align*}
\|h(A)\|_{B(X)}\le \int_{\mathbb{R}_+^n}\|T_A(r)\|_{B(X)}|b(r)|dr\le M\|b\|_{L^1}\le Mk(f,\mu).
\end{align*}
\end{proof}

\begin{theorem}\label{sufficient2}
Let $\psi\in \mathcal{T}^0_n$ with a representing measure $\mu$, and $g_t(s):=e^{t\psi(s)}=\mathcal{L}f_t(-s)$ with $f_t\in L^1(\mathbb{R}_+^n),$ $t>0$.
If 
\begin{align*}
\limsup\limits_{t\to 0+}tk(f_t,\mu)<\infty,
\end{align*}
then $\psi\in \mathcal{T}^Y_n$ (i.e. $\psi(A)$ is a generator of a quasibounded holomorphic semigroup $g(A)$ in $X$).
\end{theorem}

\begin{proof} Indeed, by Lemma \ref{lemma1}
\begin{align*}
\limsup\limits_{t\to 0+}t\left\|\frac{d}{dt}g_t(A)\right\|_{B(X)}=\limsup\limits_{t\to 0+}t\|\psi(A)g_t(A)\|_{B(X)}\le 
M \limsup\limits_{t\to 0+}tk(f_t,\mu)<\infty.
\end{align*}  
Now the result follows from Theorem \ref{criteria}.
\end{proof}

\begin{corollary}\label{corollary1}
Let in addition  $f_t\ge 0$  and   $f_t(r-u)\ge f_t(r)$ for all $r,u,r-u\in \mathbb{R}_+^n, t>0$. Let $\Gamma_u:=\mathbb{R}_+^n\setminus (u+\mathbb{R}_+^n)$.
If 
\begin{align*}
\limsup\limits_{t\to 0+}t \int_{\mathbb{R}_+^n} \int_{\Gamma_u}f(r)drd\mu(u)<\infty,
\end{align*}
then $\psi\in \mathcal{T}^Y_n$ (i.e. $\psi(A)$ is a generator of a quasibounded holomorphic semigroup $g(A)$ in $X$).
\end{corollary}

\begin{proof} It is  sufficient  to prove that
\begin{align*}
k(f_t,\mu)=2 \int_{\mathbb{R}_+^n} \int_{\Gamma_u}f_t(r)drd\mu(u).
\end{align*}
To this end, note that
\begin{align*}
 \int_{\mathbb{R}_+^n} |f_t(r-u)-f_t(r)|dr&= \int_{\Gamma_u}f_t(r)dr+\int_{\mathbb{R}_+^n\setminus \Gamma_u}(f_t(r-u)-f_t(r))dr\\[5pt]
 &= \int_{\Gamma_u}f_t(r)dr+\int_{u+\mathbb{R}_+^n}f_t(r-u)dr-\int_{\mathbb{R}_+^n\setminus \Gamma_u}f_t(r)dr\\[5pt]
 &=\int_{\Gamma_u}f_t(r)dr+\int_{\mathbb{R}_+^n}f_t(r)dr-\int_{\mathbb{R}_+^n\setminus \Gamma_u}f_t(r)dr\\[5pt]
 &=2\int_{\Gamma_u}f_t(r)dr.
 \end{align*}
 The proof is complete.
\end{proof}








\renewcommand{\refname}{\begin{center}{\Large\bf References} \end{center}}
\makeatletter
\renewcommand{\@biblabel}[1]{#1.\hfill}
\makeatother

\bibliographystyle{amsplain}

\begin{thebibliography}{99}


\fontsize{10.6pt}{4mm}{\selectfont





\bibitem{a&A}
\textsc{A. R. Mirotin}, \textit{Multivariable ${\cal T}$-calculus in generators of
$C_0$-semigroups}, {\it St. Petersburg Math. J.}, {\bf 11} (2000),
315-335.

\bibitem{umzh}
\textsc{A. R.   Mirotin},  \textit{Completely monotone functions on lie semigroups},  Ukr Math J.,  {\bf 52}, 964--968 (2000). https://doi.org/10.1007/BF02591791

\bibitem{heyer}
\textsc{H. Heyer}, \textit{Probability measures on locally compact groups}, Springer-Verlag, Berlin-Heidelberg-New York (1977).


 \bibitem{Boch}
\textsc{S. Bochner}, \textit{Harmonic analysis and the theory of probabylity}, University of California Press, Berkeley and Los Angeles (1955).



\bibitem{SSV}
\textsc{R. Shilling, R. Song, Z. Vondracek}, \textit{Bernstein functions.
Theory and applications},  de Greyter, Berlin-New York  (2010).


\bibitem{BrRob} 
\textsc{O. Bratteli, D. W. Robinson},  \textit{Operator Algebras and Quantum Statistical Mechanics 1,  $C^*$ and $W^*$-Algebras. Symmetry Groups. Decomposition of States}, Springer-Verlag, Berlin -- Heidelberg -- New York (1979). 

\bibitem{Burb}
\textsc{N. Bourbaki},
 \textit{Elements de mathematique. Livre VI. Integration. 2nd ed., Ch. 1 -- 9},   Hermann, Paris (1965 -- 1969).

\bibitem{BCR}
\textsc{Ch. Berg, J.P.R. Christensen, P. Ressel}, \textit{Harmonic
analysis on semigroups}, Grad. Texts in Math., vol. 100,
Springer-Verlag, New York-Berlin (1984).


\bibitem{CK}
\textsc{A. S. Carasso and  T. Kato}, \textit{On subordinated holomorphic
semigroups}, {\it Trans. Am. Math. Soc.} {\bf 327}, No. 2,
867--878  (1991).


\bibitem{Yos59}
\textsc{K. Yosida},  \textit{Abstract analyticity in time for solutions of a
diffusion equation}, {\it Proc. Japan. Acad.} {\bf 35},109--113 (1959).

\bibitem{Yos60}
\textsc{K. Yosida}, \textit{ Fractional powers of infinitesimal generators and the analyticity of the
semi-groups generated by them}, Proc. Japan Acad. \textbf{36}, 86--89  (1960).


\bibitem{Fuj}
\textsc{J. Fujita},  \textit{A sufficient condition for the Carasso-Kato theorem},
{\it Math. Ann.} {\bf 297} (1993), 335--341 ; 'Erratum to:  A
sufficient condition for the Carasso-Kato theorem',{\it Math.
Ann.} {\bf 299}, 390  (1994).


\bibitem{HiF}
\textsc{E. Hille and R. Phillips}, \textit{Functional Analysis and Semigroups}, Amer. Math. Soc., Providence,
R.I. (1957).

\bibitem{OaM}
\textsc{A. R. Mirotin}, \textit{Bernstein functions of several semigroup generators on Banach spaces under bounded perturbations}, Operators and Matrices, \textbf{11}, No 1,  199--217 (2017),
doi:10.7153/oam-11-14No. 



\bibitem{OaM2}
\textsc{A. R. Mirotin}, \textit{Bernstein functions of several semigroup generators on Banach spaces under bounded perturbations, II}, Operators and Matrices, \textbf{12},  No.  2, 445--463  (2018).

\bibitem{MirSF}
\textsc{A. R. Mirotin},  \textit{Criteria for Analyticity of Subordinate Semigroups}, Semigroup Forum,  \textbf{78}, No. 2,  262--275 (2009).


\bibitem{RM}
\textsc{A. R. Mirotin}, \textit{On some functional calculus of closed operators on Banach space. III.
Some topics in perturbation theory},  Russian Math., \textbf{12},   \textbf{61}, No. 12, 
19--28 (2017). 


\bibitem{MirSMZ}
\textsc{A.R. Mirotin}, \textit{On the $\mathcal{T}$-calculus of generators for $C_0$-semigroups},  Sib. Math. J., \textbf{39}, No. 3, 493--503 (1998).




\bibitem{Mir97}
\textsc{A.R. Mirotin}, \textit{Functions from the Schoenberg class $\mathcal{T}$ on the cone of dissipative elements of
a Banach algebra}, Math. Notes, \textbf{61}, No. 3--4, 524--527 (1997).


\bibitem{Mir98}
\textsc{A.R. Mirotin}, \textit{Functions from the Schoenberg class $\mathcal{T}$ act in the cone of dissipative elements
of a Banach algebra, II}, Math. Notes,
\textbf{64},  No. 3--4, 364--370 (1998).


\bibitem{Mir99}
\textsc{A.R. Mirotin}, \textit{Multidimensional $\mathcal{T}$-calculus for generators of $C_0$-semigroups},  St. Petersburg Math. J., \textbf{11},  No. 2, 315--335  (1999).

\bibitem{SMZ2011}
\textsc{A.R. Mirotin}, \textit{On some properties of the multidimensional Bochner-Phillips functional calculus}, Siberian Mathematical Journal
 \textbf{52}, {\bf  6},  1032--1041  (2011).

\bibitem{IZV2015}
\textsc{A.R. Mirotin}, \textit{On joint spectra of families of unbounded operators},  Izvestiya: Mathematics, \textbf{79}, No. 6, 1235--1259 (2015).


\bibitem{mz}
\textsc{A.R. Mirotin},  \textit{Properties of Bernstein functions of several complex variables},   Math. Notes, \textbf{93},  No. 2 (2013).


\bibitem{Mir09}
\textsc{A.R. Mirotin}, \textit{On multidimensional Bochner-Phillips functional calculus}, Probl. Fiz. Mat.
Tekh., \textbf{1},  No. 1, 63--66  (2009) (Russian).



\bibitem{EngNag}
\textsc{K.-J. Engel, R. Nagel},  \textit{A Short Course on Operator Semigroups},  Springer (2006)



\bibitem{Pinchuk}
\textsc{S. I. Pinchuk},   \textit{A boundary uniqueness theorem for holomorphic functions of several complex variables}, Mathematical Notes of the Academy of Sciences of the USSR, \textbf{15}, 116--120 (1974). https://doi.org/10.1007/BF02102390

}







\end{thebibliography}

\end{document}